\documentclass[11pt]{amsart}
\usepackage{comment}
\usepackage{amsmath,amssymb,mathrsfs, xcolor, geometry}
\usepackage{hyperref}
\hypersetup{linkbordercolor = {white}, citebordercolor = {white},}

\usepackage[alphabetic]{amsrefs}

\usepackage[alphabetic]{amsrefs}

\usepackage{float}

\numberwithin{equation}{section}

\theoremstyle{definition}
\newtheorem{thm}[equation]{Theorem} 
\theoremstyle{definition}
\newtheorem{definition}[equation]{Definition}
\newtheorem{prop}[equation]{Proposition} 
\newtheorem{cor}[equation]{Corollary}

\newtheorem{lemma}[equation]{Lemma}
\newtheorem{remark}[equation]{Remark}

\newcommand{\pair}[1]{\langle #1\rangle}
\newcommand{\set}[1]{\left\{#1\right\}}

\newcommand{\pr}[1]{\left(#1\right)}

\newcommand{\bb}[1]{\mathbb{#1}} \newcommand{\td}[1]{\widetilde{#1}}

\newcommand{\tr}{\text{tr}}
\newcommand{\D}{\Delta}

\newcommand{\n}{\nabla}
\newcommand{\sbst}{\subseteq}
\newcommand{\h}{\textbf H}

\newcommand{\bd}{\partial}
\newcommand{\N}{\mathbf{n}}

\title{
A stable self-shrinking M\"{o}bius bundle in $\mathbb R^4$}

\author{Kahnrad Braxton}
\author{Tang-Kai Lee}
\author{Jonathan J. Zhu}

\address{Department of Mathematics, Johns Hopkins University, Baltimore, Maryland, USA}
\email{kbraxto6@jhu.edu}

\address{Department of Mathematics, Columbia University, New York, NY 10027, USA}
\email{leetk@math.columbia.edu,}

\address{Department of Mathematics, University of Washington, Seattle, WA, USA}
\email{jonozhu@uw.edu}

\date{\today}

\begin{document}
\begin{abstract}
We prove the linear stability of an embedded mean curvature flow shrinker in $\mathbb R^4$ with the topology of the M\"{o}bius bundle. This provides a non-flat, non-spherical stable shrinker in higher codimension, in sharp contrast with the classification of stable hypersurface shrinkers. 
Topologically, the M\"obius shrinker models the reversal of a real blow-up, suggesting that stable higher-codimension singularities may encode non-trivial topological operations under mean curvature flow.
\end{abstract}
\maketitle


\section{\bf Introduction}
\label{sec:intro}

The mean curvature flow deforms a submanifold in the direction of its mean curvature vector and provides a canonical way of simplifying its geometry. 
Even when the initial submanifold is smooth and compact, however, singularities typically develop in finite time. 
Understanding these singularities is therefore a central problem in the study of the mean curvature flow.
Shrinkers, which evolve by homothetic shrinking under the flow, arise naturally as models for such singularities.

The space of all shrinkers is known to be rather wild, and a complete classification of singularity models appears to be out of reach. 
A fundamental principle, originating in the work of Colding--Minicozzi~\cite{CM12}, is that one should instead try to understand the singularities which are stable, and hence cannot be removed by a small perturbation of the flow. 
This point of view has led to a successful theory of generic singularities for mean curvature flow in codimension one.
%
%

The situation in higher codimension is considerably less understood. 
In this paper we study an explicit two-dimensional Lagrangian shrinker in $\mathbb R^4$ constructed by Yng-Ing Lee and Mu-Tao Wang \cite{LW09}. 
Among other things, they constructed an embedded asymptotically conical surface whose underlying topology is the total space of the M\"obius line bundle, induced by the parametrization
\begin{equation}
\label{eq:mobius-defn-intro}
\pr{e^{2i\theta}\cosh(r), e^{-i\theta}\sqrt{2}\sinh(r)}\qquad \text{ for } \theta\in \bb S^1 \text{ and } r\in\bb R.
\end{equation}
We refer to this surface throughout the paper as \emph{the M\"obius shrinker}. Bernstein \cite{B26} suggested that this may be the simplest nonorientable shrinker in $\bb R^4$; see also Remark \ref{rmk:entropy}.
Our main result is the stability of this M\"obius shrinker in the sense of the Gaussian area functional modulo translations and dilation; that is, the $F$-stability as defined in \cite{CM12} (cf. \cite{AS13,ALW,LL15}).

\begin{thm}
	\label{thm:main}
	The M\"obius shrinker is $F$-stable.
\end{thm}

Theorem~\ref{thm:main} is in stark contrast with the codimension one picture. 
There, apart from the flat shrinker, $F$-stability forces a complete embedded shrinker to be the round sphere.
The M\"obius shrinker shows that this rigidity breaks down already for surfaces in $\mathbb R^4.$
Theorem~\ref{thm:main} suggests that the M\"obius shrinker should be regarded as a natural candidate for a generic singularity model in higher codimension mean curvature flow. 
The $F$-stability by itself does not establish such a genericity statement. 
Nevertheless, the result points to a phenomenon that has no direct analogue in the hypersurface generic-singularity picture. 
Generic higher-codimension flows may be able to perform nontrivial topological operations which are invisible in codimension one.

There is a particularly simple topological interpretation of this possibility. 
The total space of the M\"obius line bundle can be identified with the real blow-up
$\text{Bl}^{\mathbb R}_0\mathbb R^2$
of the plane at the origin; that is, the origin is replaced by the exceptional divisor $\bb{RP}^1$. 
Correspondingly, blowing up a point on a closed surface in the real category amounts topologically to taking the connected sum with $\bb{RP}^2$. 
From this perspective, a mean curvature flow developing a M\"obius-shrinker singularity should be thought of as \emph{undoing a real blow-up}, or equivalently, removing an $\bb{RP}^2$ summand. 
Theorem~\ref{thm:main} suggests that such an operation is not ruled out by linear stability and may therefore be a robust topological change available to higher codimension mean curvature flow. 
This is quite different from the codimension-one program, where generic singularities are expected to be governed by the much more rigid spherical and cylindrical models \cite{CM12,CCMS,SX1,SX2,CCS, BK23}.
Note that by \cite{LZ-tangent}, at such a singular time-slice, the tangent cone at the singularity is the Schoen--Wolfson $(2,1)$-cone, and understanding topological changes through such a M\"obius singularity will contribute to the program in high codimensions.
For the case of almost calibrated Lagrangian mean curvature flow, see, for example, \cite{TY02, J15, L25}, where Lawlor necks and some expanding solutions \cite{JLT} are expected to be models for a surgery process resolving a neckpinch singularity; see Remark~\ref{rmk:Lawlor}.




A natural next question is whether the M\"obius shrinker \eqref{eq:mobius-defn-intro} actually occurs as a singularity model of a mean curvature flow starting from a compact submanifold. 
In the present setting, one may ask more specifically whether there exists a compact surface in $\mathbb R^4$ whose flow develops a singularity modeled on the M\"obius shrinker. 
Results constructing compact flows with prescribed asymptotically conical singularity models in codimension one \cite{LZ-closed} suggest that such a realization problem is approachable. 
A particularly appealing possibility is that within a suitable class of initial data, a flow starting from a Klein bottle develops such a singularity and thereby removes one of its $\bb{RP}^2$ factors.
Understanding whether this behavior occurs, and whether it is stable under perturbations of the initial surface, would provide a nonlinear counterpart to Theorem~\ref{thm:main}.

There is a parallel story in Ricci flow. 
Feldman--Ilmanen--Knopf constructed a four-dimensional shrinking K\"ahler--Ricci soliton on the total space of $\mathcal O_{\mathbb{CP}^1}(-1),$ which is the complex blow-up~$\text{Bl}^{\mathbb C}_0\mathbb C^2$ of $\mathbb C^2$ at the origin \cite{FIK03}. 
This shrinker models the contraction of the exceptional $\mathbb{CP}^1$ and hence the operation of undoing a complex blow-up. 
Recently, Naff--Ozuch \cite{NO25} proved that the four-dimensional FIK blow-down shrinker is linearly stable \cite{NO25}. 
Their result suggests that undoing complex blow-ups should occur as a robust topological operation in four-dimensional Ricci flow, as Theorem~\ref{thm:main} does for the real blow-down described above. 
The FIK shrinker is carried by the complex blow-up $\text{Bl}^{\mathbb C}_0\mathbb C^2,$
whereas the M\"obius shrinker is carried by the real blow-up $\text{Bl}^{\mathbb R}_0\mathbb R^2.$
Thus, the M\"obius shrinker may be viewed as a mean curvature flow counterpart of the FIK blowdown shrinker. 
In both cases, a stable noncompact shrinking soliton encodes the reversal of a basic blow-up operation. 
This analogy was also one of the motivations for the present work, together with the appearance of real blow-ups; see Remark~\ref{rmk:LNZ-construction}.
It raises the broader question of whether higher codimension mean curvature flow admits a topological surgery picture in which stable noncompact shrinkers describe elementary operations in much the same way that shrinking Ricci solitons are expected to do in Ricci flow.

There exist minimal surfaces and free-boundary minimal surfaces whose topology is related to the M\"obius bundle.
See \cite{PGAM} and \cite{MM25}.


We briefly describe the proof ideas.
The $F$-stability means the shrinker is stable for the Gaussian area functional modulo translations and a dilation, which give explicit unstable directions listed in Table~\ref{tab:evectors}.
The second variation of the Gaussian functional gives rise to an index form, which is computed in Section~\ref{sec:index-form}. 
Using the Fourier basis fixed in Section~\ref{sec:prelim}, we decompose this index form into a family of one-dimensional quadratic forms so that the stability problem reduces to estimating each Fourier mode, indexed by $m\in\bb N_0$, separately, which is carried out in the following sections.

Sections \ref{sec:34-mode} estimates the index form for $m=3$ and $4.$
The $m=3$ mode provides the model for much of the argument, in which an explicit Jacobi field allows us to make a ground-state-type transformation, rewriting the index form in terms of the ratio between an arbitrary normal field and the Jacobi field. 
This immediately gives non-negativity, together with a precise characterization of the equality case. 
The same idea makes the modes $m\ge4$ comparatively straightforward; moreover, the positive angular contribution becomes stronger as $m$ increases, so the higher modes follow from the estimates at the borderline modes.

The modes $m\le 2$ are more delicate because the geometric unstable directions begin to occur.
In Section~\ref{sec:12-mode}, we first identify these directions explicitly and separate them from the remaining variations using the orthogonality conditions appearing in the definition of $F$-stability. 
After suitable changes of variables, the resulting quadratic forms are reduced to weighted one-dimensional estimates of Sturm--Liouville type, together with explicit non-negative remainder terms. 
These estimates show that, once the translation and dilation directions are removed, every low Fourier mode is non-negative, and the rigidity statements in the individual mode estimates then identify all possible zero directions.
Section~\ref{sec:0-mode} deals with the lowest mode, which uses a Sturm--Liouville argument and other techniques.

We complete the proof of $F$-stability in Section~\ref{sec:final}. 
Beyond Theorem~\ref{thm:main}, the mode analysis also yields a complete classification of the unstable modes and Jacobi fields, which turn out to be precisely those arising from the explicit geometric variations listed in Table~\ref{tab:evectors}. 
We conclude with several remarks, including possible directions for further investigation.

\subsection*{Acknowledgements}
The authors thank Jacob Bernstein, Yng-Ing Lee, and Ao Sun for helpful suggestions.
KB would like to thank his advisor, Jacob Bernstein, for introducing him to this problem and for his continued guidance and valuable insights into the methods and background underlying it. 
TKL and JJZ thank Keaton Naff and Jingze Zhu for many inspiring discussions and suggestions, without which this project would not have taken its present form.
Part of the project was initiated when TKL was with Keaton Naff and Jingze Zhu at an event commemorating the life and work of
Richard Hamilton, and he thanks Simon Brendle for organizing the event.


\subsection*{AI declaration}
After obtaining the results in Sections \ref{sec:12-mode} and \ref{sec:mode-theta}, ChatGPT was used to check the validity of some routine arguments and to suggest certain simplifications of the existing arguments, precisely, the additional substitutions in \eqref{d0-m=1} and \eqref{m-2-sug}, and the Brascamp--Lieb approach in Proposition \ref{prop:mode-0-theta-analysis}. 
AI was not involved in any other mathematical analysis in this article. 
This article does not contain LLM-generated text.

\section{\bf Preliminaries}
\label{sec:prelim}

\subsection{M\"obius shrinker}

In \cite{LW09}, Y.-I. Lee and M.-T. Wang constructed a family of Lagrangian shrinkers that are asymptotic to the Schoen--Wolfson $(p,q)$-cones $C_{p,q}$ for any coprime $p$ and $q.$
We work on the case when $p=2$ and $q=1,$ which gives rise to the shrinker $\td X\colon \bb S^1\times \bb R\to\bb R^4\simeq \bb C^2$ defined by\footnote{The one constructed in \cite{LW09}*{Section 2.2} has an additional $i$ factor in the second complex component.}
\begin{align}\label{eq:F-cyl-parametrization}
\td X(\theta,r)=\pr{e^{2i\theta}\cosh(r), e^{-i\theta}\sqrt{2}\sinh(r)}.
\end{align}
The symmetry $\td X(\theta+\pi,-r)=\td X(\theta,r)$ implies that $\td X$ descends to a map 
\begin{align}\label{eq:M-embedding}
X\colon M\to\bb R^4
\end{align}
where $M=\pr{\bb S^1\times \bb R}/\sim$ and the equivalence relation is defined by $(\theta+\pi,-r)\sim(\theta,r).$
Thus, $M$ is a M\"obius bundle and it is straightforward to check that $X$ is an embedding of $M$ into $\mathbb{R}^4$. 
In the rest of the note, we will call this embedding or its image (also identified with $M$) the M\"obius shrinker.

\subsubsection{Geometry of the M\"obius shrinker}
\label{sec:shrinker-geometry}
We work with the parametrization $X\colon M\to\bb R^4$ given as above by $X(\theta,r)=(e^{2i\theta}\cosh(r), e^{-i\theta}\sqrt{2}\sinh(r)),$ on which there are locally defined tangent vectors 
\begin{align*}
X_\theta&= \pr{2i e^{2i\theta}\cosh(r), -\sqrt 2 i e^{-i\theta}\sinh(r)}\,\, \text{ and }\,\,
X_r= \pr{e^{2i\theta}\sinh(r), \sqrt 2 e^{-i\theta}\cosh(r)}.
\end{align*}
The Lagrangian condition allows us to obtain two locally defined normal vectors
\begin{align*}
\N_1 = JX_\theta = (-2 e^{2i\theta} \cosh(r), \sqrt 2 e^{-i\theta}\sinh(r)) 
\,\,\text{ and }\,\,
\N_2=JX_r = (ie^{2i\theta}\sinh(r), \sqrt 2 ie^{-i\theta}\cosh(r))
\end{align*}
which satisfy
\begin{align}\label{eq:n-boundary}
\N_1(\theta+\pi, -r)
= \N_1(\theta, r)
\,\,\text{ and }\,\,
\N_2(\theta+\pi, -r)
= -\N_2(\theta, r)
\end{align}
for any $(\theta,r)\in \bb S^1\times\bb R.$
Straightforward calculations lead to 
\begin{align*}
g := X^* g_{\rm euc}
= 2\omega(r)\,d\theta\otimes d\theta
+ \omega(r)\, dr\otimes dr
\end{align*}
where $\omega(r) := g_{rr}
= \sinh^2r+2\cosh^2r,$
and 
\begin{align*}
A_{\theta\theta}
= \frac 1{\omega(r)} \pr{4\cosh^2(r) - \sinh^2(r)}\N_1,\,\,
A_{rr}
= \frac {-1}{\omega(r)} \, \N_1,\,\,\text{ and }\,\,
A_{\theta r}
= \frac {-2}{\omega(r)}\, \N_2
\end{align*}
where $A(Y,Z)=\n^\perp_Y Z$ is the second fundamental form of $M$ in $\bb R^4.$
Combining these implies that the mean curvature vector of $M$ is 
\begin{align*}
\h = \tr_g A
= -\frac 12 X^\perp
& = \frac 1{2 \omega} \N_1.
\end{align*}
In particular, the mean curvature is non-vanishing.

As observed in \cite{LW09}, the shrinker $M$ is an asymptotically conical shrinker.
The link of its asymptotic cone is the curve
\begin{align*}
\frac 1{\sqrt 3} \set{\pr{e^{2it}, \sqrt 2 e^{it}}: t\in[0,2\pi]}
\sbst \bb S^3.
\end{align*}
The Lagrangian angle of $M$ can locally be represented by $\theta$ and hence it is a Hamiltonian stationary shrinker.
Note that the M\"obius shrinker is not a zero-Maslov class Lagrangian.

\subsubsection{Normal vector fields and Fourier expansions}
\label{sec:normal-fourier}

The stability problem is about deformations of the shrinker, on which normal vector fields provide the infinitesimal information.
Let $U$ be a smooth normal vector field on the shrinker $M,$ i.e., a section of the normal bundle $NM.$
Following the notations in Section~\ref{sec:shrinker-geometry}, we can find smooth functions $u$ and $v$ on $M$ such that
\begin{align*}
U = \frac{u}{g_{\theta\theta}}\N_1 + \frac{v}{g_{rr}}\N_2
= \frac{u}{g_{\theta\theta}} JX_\theta + \frac{v}{g_{rr}} JX_r.
\end{align*}
By the construction of $M$ and the boundary conditions \eqref{eq:n-boundary}, the functions $u$ and $v$ have to satisfy
\begin{align}
\label{eq:u-v-boundary}
u(\theta,r) = u(\theta+\pi, -r)
\,\,\text{ and }\,\,
v(\theta,r) = -v(\theta+\pi, -r)
\end{align}
so that it is a globally deined expression of a normal vector field.
Each of $u$ and $v$ can be expanded by their Fourier modes as
\begin{equation}\label{eq:uv-expansion}
\begin{split}
u&= \sum_{m=0}^\infty \pr{\sqrt{2}\cos(m\theta)\alpha^1_m(r)+\sqrt{2}\sin(m\theta)\alpha^2_m(r)}
\,\,\text{ and}\\
v&= \sum_{m=0}^\infty \pr{\cos(m\theta)\beta^1_m(r)+\sin(m\theta)\beta^2_m(r)}
\end{split}
\end{equation}
for some smooth functions $\alpha^j_m$'s and $\beta^j_m$'s on $\bb R.$
The boundary conditions \eqref{eq:u-v-boundary} imply that 
\begin{align}
\label{eq:m-even-parity}
\text{when }m\text{ is even},
\alpha^1_m\text{ and }\alpha^2_m
\text{ are even and }
\beta^1_m\text{ and }\beta^2_m
\text{ are odd}
\end{align}
and that
\begin{align}
\label{eq:m-odd-parity}
\text{when }m\text{ is odd},
\alpha^1_m\text{ and }\alpha^2_m
\text{ are odd and }
\beta^1_m\text{ and }\beta^2_m
\text{ are even}.
\end{align}

Special normal vector fields correspond to certain unstable direction that will be ruled out in the $F$-stability that we will introduce next.
For that purpose, we list all the (weakly) unstable directions that can be obtained in a straightforward way.
The following table indicates the non-zero modes of each unstable deformation, where we view $\bb C^2\simeq \bb R^4$ in the canonical way and let $V_{ij} = x_i e_j - x_j e_i$ be the generator of rotation in the $x_ix_j$-plane. 


\begin{table}[H]
    \centering
    \begin{tabular}{c|c|c|c}
        Transformation & Generator & Mode $m$ & Nonzero Fourier coefficients  \\
        \hline
        Dilation & $x^\perp$ & 0 & $\alpha_0^1 = -\sqrt{2}$ \\ 
        Rotation & $2V_{12}^\perp =  V_{34}^\perp$ & 0 & $\beta^1_0 = \sinh(2r)$\\

        Translation & $e_3^\perp$ & 1 & $\alpha^1_1=\sinh(r), \beta^2_1 = \sqrt{2}\cosh(r)$ \\
        Translation & $e_4^\perp$ & 1 & $\alpha^2_1 = -\sinh(r), \beta^1_1 = \sqrt{2}\cosh(r)$\\
        Rotation & $\frac{1}{2}(V_{13}^\perp - V_{24}^\perp)$ & 1 & $\alpha^1_1 = \frac{3}{4}\sinh(2r), \beta^2_1 =-\frac{\sqrt{2}}{2}$\\
        Rotation & $\frac{1}{2}(V_{14}^\perp + V_{23}^\perp)$ & 1 & $\alpha^2_1 = \frac{3}{4}\sinh(2r), \beta^1_1 =\frac{\sqrt{2}}{2}$\\
        
        Translation & $e_1^\perp$ & 2 & $\alpha^1_2 = -\sqrt{2}\cosh(r), \beta^2_2 = -\sinh(r)$ \\
        Translation & $e_2^\perp$ & 2 & $\alpha^2_2 = -\sqrt{2}\cosh(r), \beta^1_2 = \sinh(r)$; \\

        Rotation & $\frac{1}{2}(V_{13}^\perp + V_{24}^\perp)$ & 3 & $\alpha^1_3 = \frac{3}{4}\sinh(2r), \beta^2_3 =\frac{\sqrt{2}}{2}\cosh(2r)$\\
        Rotation & $\frac{1}{2}(-V_{14}^\perp + V_{23}^\perp)$ & 3 & $\alpha^2_3 = \frac{3}{4}\sinh(2r), \beta^1_3 =-\frac{\sqrt{2}}{2}\cosh(2r)$\\
    \end{tabular}\vspace{6pt}
    \caption{Weakly unstable directions for $M.$}
    \label{tab:evectors}
\end{table}

Note that all translations and rotations are accounted for, as from the circle symmetry of the M\"{o}bius shrinker $M,$ we have $2V_{12}^\perp = V_{34}^\perp$. 

\subsection{$F$-stability}

In this section, we recall the Colding--Minicozzi theory on stability properties of shrinkers \cite{CM12}, whose high codimensional version was studied in \cite{AS13,ALW,LL15}.
Given an $n$-dimensional submanifold $\Sigma$ in $\bb R^N,$ its $F$-functional or Gaussian area is defined to be 
\begin{align*}
F(\Sigma):= \frac 1{(4\pi)^{n/2}} \int_\Sigma e^{-|x|^2/4} dx,
\end{align*}
and more generally, for any $x_0\in\bb R^N$ and $t_0>0,$ 
\begin{align*}
F_{x_0,t_0}(\Sigma)
:= \frac 1{(4\pi t_0)^{n/2}} \int_\Sigma e^{-|x-x_0|^2/(4t_0)} dx.
\end{align*}
Note that $\Sigma$ is a shrinker, or $\Sigma$ satisfies $\h=-x^\perp/2,$ if and only if $\Sigma$ is a critical point of the functional $F=F_{0,1}.$

\begin{definition}
    Let $\Sigma$ be an $n$-dimensional shrinker in $\bb R^N.$
    We say $\Sigma$ is {\bf $F$-stable} if for any compact variation $\Sigma_s$ of $\Sigma,$ there exist variations $x_s$ of $0$ and $t_s$ of $1$ such that $\bd_s^2 F_{x_s,t_s}(\Sigma_s)|_{s=0} \ge 0.$
\end{definition}

\begin{remark}
\label{rmk:weighted-spaces}
    If $\Sigma$ is a shrinker, by the calculations in \cite{CM12} (cf. \cite{AS13,ALW,LL15}), for a compact variation $\Sigma_s$ of $\Sigma$ given by a normal vector $V,$ if we consider the index form 
    \[\pr{4\pi}^{n/2} Q(V):= \int_\Sigma \left(|\nabla^\perp V|^2
    - g^{ik}g^{j\ell} \pair{A_{ij}, V}\langle A_{k\ell},V\rangle
    -\frac{1}{2}|V|^2 \right) e^{-|x|^2/4}\,dx, \]
    then for any variation $x_s$ of $0$ with $\bd_sx_s|_{s=0}=y$ and $t_s$ of $1$ with $\bd_st_s|_{s=0}=\tau,$ it follows that
    \begin{align*}
    \bd_s^2 F_{x_s,t_s}(\Sigma_s)|_{s=0}
    = Q(V)
    + \frac 1{(4\pi)^{n/2}}
    \int_\Sigma \pr{
    - 2\tau\pair{V,\h} 
    - \tau^2|\h|^2
    + \pair{V,y}
    - \frac 12|y^\perp|^2
    }e^{-|x|^2/4}\,dx.
    \end{align*}
    By \cite{LL15}, it then follows that $\Sigma$ is $F$-stable if and only if $Q(V)\ge 0$ for any normal $V$ with both $V$ and $\n^\perp V$ of polynomial growth such that $V$ is $L^2$-orthogonal to $\h$ and the all translation directions.
    When analyzing the form $Q,$ we will work with normal vector fields that are $H^1$ with respect to the Gaussian weight, which will form a slightly larger space.
    As later we will work on different Fourier modes, we consider 
    \begin{align}\label{eq:def-HW}
    H^1_W(\mathbb{R}) = \set{f \in L^2_\rho(\mathbb{R}) : f' \in L^2_{\rho/\omega}(\mathbb{R}) }
    \end{align}
    where $\rho = e^{-|F|^2/4} = \exp(-(3\cosh(2r)-1)/8)$ is the Gaussian weight on $M$ and $L^2_h(\bb R)$ is the space of function on $\bb R$ that is $L^2$ with respect to $h\,dx.$
    Integration by parts make sense when taking weighted integrals of such functions by the standard cut-off argument and the dominated convergence theorem.
\end{remark}

A closely related notion is the entropy-stability \cite{CM12}, which considers all possible deformations (instead of just compactly supported ones) and is potentially stronger than the $F$-stability for a shrinker that does not split off a line isometrically.
Theorem~\ref{thm:main} implies that the entropy of the M\"obius shrinker $M$ can not decrease if one deforms $M$ by a weighted integrable variation. 
See Remark~\ref{rmk:entropy} for more on entropy.

\section{\bf Calculating index form}
\label{sec:index-form}

In this section, we decompose the index form on the M\"{o}bius shrinker $M$ according to the Fourier decomposition of normal sections as in Section \ref{sec:normal-fourier}. 
We first substitute a general normal section into the index form.
Recall that we define 
\begin{align}\label{convention-omega-rho}
\omega = g_{rr} = \frac{3\cosh(2r)+1}{2}
\,\,\text{ and }\,\,
\rho = e^{-|F|^2/4} = \exp(-(3\cosh(2r)-1)/8).
\end{align}
The volume element is $\sqrt{2}\omega\, d\theta\, dr$.

\begin{lemma}
    \label{lem:index-form-mobius}
    Let $M$ be the M\"{o}bius shrinker and $U = \frac{u}{g_{\theta\theta}}\N_1 + \frac{v}{g_{rr}}\N_2$. The index form is given by
    \[
    \begin{split}
    4\pi Q(U)= \int_M \Bigg(& 
    \omega^{-2}\pr{\frac{1}{4}u_\theta^2+\frac{1}{2}u_r^2 + \frac{1}{2}v_\theta^2+v_r^2}
    \\& -\frac{3\sinh(2r)}{2\omega^3}(u(u_r+v_\theta) +v(2v_r-u_\theta))
    \\& +\frac{1}{16\omega^4}(-7+3\cosh(2r))(11+9\cosh(2r))u^2   -\frac{1}{4\omega}u^2
    \\& + \frac{1}{4\omega^4}(-25+9\cosh(4r))v^2 -\frac{1}{2\omega}v^2
    \Bigg)\rho\,dx.
    \end{split}
    \]
\end{lemma}

\begin{proof}
Recall that
\[4\pi Q(U)= \int_M \left(|\nabla^\perp U|^2
- g^{ik}g^{j\ell} \pair{A_{ij}, U}\langle A_{k\ell},U\rangle
-\frac{1}{2}|U|^2 \right) \rho\, dx.\]
Recall also that $|\N_1|^2 = g_{\theta\theta}=2\omega$ and $|\N_2|^2=g_{rr}=\omega$, so
\[\frac{1}{2}|U|^2=\frac{u^2}{4\omega} + \frac{v^2}{2\omega}.\]
We may compute $\nabla^\perp_U$ by projecting $DU$ to the normal bundle. Namely, we find that
\[
    \nabla^\perp_\theta U = \left(u_\theta + \frac{3v}{\omega}\sinh(2r)\right)\frac{\N_1}{2\omega}+ \left(v_\theta -\frac{3u}{2\omega}\sinh(2r) \right)\frac{\N_2}{\omega},
\]
\[
    \nabla^\perp_r U = \left(u_r - \frac{3u}{2\omega}\sinh(2r)\right)\frac{\N_1}{2\omega}+ \left(v_r -\frac{3v}{2\omega}\sinh(2r)\right)\frac{\N_2}{\omega}.
\]
It follows that 
\[\begin{split}
    |\nabla^\perp U|^2 &= g^{\theta\theta}|\nabla^\perp_\theta U|^2 + g^{rr}|\nabla^\perp_r U|^2
    \\&= \frac{1}{\omega^2}\left(\frac{1}{4} u_\theta^2  + \frac{3}{2\omega}\sinh(2r) u_\theta v +\frac{9}{4\omega^2}\sinh^2(2r)v^2 + \frac{1}{2} u_r^2 - \frac{3}{2\omega}\sinh(2r) u_r u +\frac{9}{8\omega^2}\sinh^2(2r)u^2   \right)
    \\& \qquad + \frac{1}{\omega^2}\left( \frac{1}{2}v_\theta^2  - \frac{3}{2\omega}\sinh(2r)uv_\theta + \frac{9}{8\omega^2}\sinh^2(2r)u^2 + v_r^2 -\frac{3}{\omega}\sinh(2r)v v_r  + \frac{9}{4}\sinh^2(2r)v^2 \right)
\end{split}
\]
Finally, the curvature term is 
\[
\begin{split}
  -g^{ik}g^{j\ell} \pair{A_{ij}, U}\pair{A_{k\ell},U}
  &= g^{\theta\theta}g^{\theta\theta} \pair{A_{\theta\theta}, U}^2 + 2g^{\theta\theta}g^{rr} \pair{A_{\theta r}, U}^2 + g^{rr}g^{rr} \pair{A_{rr}, U}^2
  \\&=  -\frac{1}{\omega^2}\left(  (5+3\cosh(2r))^2\frac{u^2}{16\omega^2} +  4\frac{v^2}{\omega^2} + \frac{u^2}{\omega^2} \right).
\end{split}
\]
Collecting terms yields the claimed formula for $Q(U)$. 
\end{proof}

We now expand the above in our Fourier basis as in \eqref{eq:uv-expansion} for some smooth functions $\alpha^j_m$ and $\beta^j_m$ on $\bb R.$
Substituting and integrating in $\theta$ leads to the following simplified form.

\begin{prop}
\label{prop:form-fourier}
For $U = \frac{u}{g_{\theta\theta}}\N_1 + \frac{v}{g_{rr}}\N_2$ with the decomposition in \eqref{eq:uv-expansion},
we have the splitting 
\[2\sqrt{2} Q(U)= Q_0^\theta(\alpha_0) + Q_0^r(\beta_0)+\sum_{m=1}^\infty \left( Q_m(\alpha^1_m,\beta^2_m) + Q_m(\alpha^2_m,-\beta^1_m) \right), \]
where the quadratic forms\footnote{We emphasize that $Q_m$ is a quadratic form acting on one element $(\alpha,\beta).$}
$Q_m$'s are defined by
\begin{equation}
\label{eq:Q-m}
\begin{split}
    Q_m(\alpha,\beta)
    = \int_0^\infty
    \Bigg[
    &\frac{(\alpha')^2+(\beta')^2}{\omega}
    +\left(-1+\frac{m^2}{2\omega}\right)\alpha^2\\
    &+\left(
    -1+\frac{m^2+1}{2\omega}
    +\frac{2}{\omega^2}
    \right)\beta^2
    -\frac{3\sqrt2\,m\sinh(2r)}{\omega^2}\alpha\beta
    \Bigg]\rho\,dr,
\end{split}
\end{equation}
\begin{align*}
Q_0^\theta(\alpha)
= \int_0^\infty
\left(
\frac{(\alpha')^2}{\omega}
-\alpha^2
\right)\rho\,dr
,\quad\text{and}\quad
Q_0^r(\beta)
= \int_0^\infty
\left(
\frac{(\beta')^2}{\omega}
+\left(
-1+\frac{2}{\omega^2}
\right)\beta^2
\right)\rho\,dr.
\end{align*}
\end{prop}

\begin{proof}
Using that Fourier modes of different frequencies are orthogonal, one can verify
\[
\begin{split}
    \frac{1}{2}\int_0^{2\pi} u^2 \,d\theta &= \pi\sum_{m=0}^\infty  ((\alpha^1_m)^2 + (\alpha^2_m)^2)   
    ,\qquad 
    \int_0^{2\pi} v^2 \,d\theta = \pi\sum_{m=0}^\infty  ((\beta^1_m)^2 + (\beta^2_m)^2),   \\
    \int_0^{2\pi} \left( \frac{1}{4}u_\theta^2 + \frac{1}{2}v_\theta^2 \right) \,d\theta &= \frac{\pi}{2}\sum_{m=0}^\infty  m^2((\alpha^1_m)^2 + (\alpha^2_m)^2 + (\beta^1_m)^2 + (\beta^2_m)^2),  \\
    \int_0^{2\pi} \left( \frac{1}{2}u_r^2 +v_r^2 \right) \,d\theta &= \pi\sum_{m=0}^\infty  m^2(((\alpha^1_m)')^2 + ((\alpha^2_m)')^2 + ((\beta^1_m)')^2 + ((\beta^2_m)')^2),  \\
    \int_0^{2\pi} \left( uu_r +2vv_r \right) \,d\theta &= \pi\sum_{m=0}^\infty  (((\alpha^1_m)^2)' + ((\alpha^2_m)^2)' + ((\beta^1_m)^2)' + ((\beta^2_m)^2)'),  \\
    \int_0^{2\pi} \left( uv_\theta - u_\theta v \right) \,d\theta &= 2\sqrt{2}\pi\sum_{m=0}^\infty  m(\alpha^1_m \beta^2_m - \alpha^2_m\beta^1_m).
\end{split}
\]
We then substitute these identities into Lemma \ref{lem:index-form-mobius}, recalling that the volume form in these coordinates is $\sqrt{2}\omega\, d\theta\, dr$. 
Integrating by parts on the $((\alpha^1_m)^2)' + ((\alpha^2_m)^2)' + ((\beta^1_m)^2)' + ((\beta^2_m)^2)'$ term and collecting the remaining terms yields the claimed formula. 
For $m>0$, one can observe that swapping $(\alpha^1_m, \beta^2_m)$ for $(\alpha^2_m, -\beta^1_m)$ reverses their roles. For $m=0$, there are only the Fourier coefficients $\alpha_0,\beta_0$, and there is no cross term as the last identity vanishes.
\end{proof}

\begin{remark}
    From the integral identities above, we also note that \[2\sqrt{2}\pi \|U\|^2_{L^2_W} = \sum_{m=0}^\infty \int_0^\infty ((\alpha^1_m)^2 + (\alpha^2_m)^2 + (\beta^1_m)^2 + (\beta^2_m)^2 \rho\, dr.\]
    By the splitting above, henceforth we need only analyze the forms $Q_m$ on pairs $(\alpha,\beta)$ relative to the $\rho$-weighted $L^2$ norm $\int_0^\infty (\alpha^2+\beta^2)\rho \, dr$. 
\end{remark}

\section{\bf Modes $3$ and $4$}
\label{sec:34-mode}

A feature of the linear stability problem for mean curvature flow is that we have the explicit eigenvectors as in Table \ref{tab:evectors}. 
We first describe a convenient reparametrisation given a positive (for both functions) eigenvector. This lemma will also be useful for other modes, but essentially completes the analysis for $m=3$ and $4$, so we include it here. 
See \cite{B27} for a more general formulation; we have stated the lemma below in a form most applicable for our purposes here. 
We will keep using the notations in \eqref{convention-omega-rho}.

\begin{lemma}
    \label{lem:pos-eigenvector}
    Consider the operator $\mathcal{L}$ on $H^1_W(\mathbb{R})\oplus H^1_W(\mathbb{R})$ defined by 
    $$\mathcal{L}= \frac{1}{\rho}\frac{d}{dr}\left(\frac{\rho}{\omega} \frac{d}{dr}\right) = \frac{1}{\omega}\frac{d^2}{dr^2} +\left(\frac{\rho'}{\rho}-\frac{\omega'}{\omega}\right)\frac{d}{dr}$$ 
    and let $B = \begin{pmatrix} B_1 & -B_3 \\ -B_3 & B_2\end{pmatrix}$ be a smooth, symmetric, $L^2_\rho$-integrable
    matrix-valued function on $\mathbb{R}$.
    Suppose $(\alpha,\beta),(\tilde{\alpha},\tilde{\beta})\in H^1_W(\mathbb{R})\oplus H^1_W(\mathbb{R})$ are both even-odd (resp. odd-even) pairs, and $(\tilde{\alpha},\tilde{\beta})$ is additionally a solution of 
    \begin{align*}
    \begin{cases}
    \mathcal L \td \alpha = B_1\td\alpha - B_3\td\beta\\
    \mathcal L \td \beta = -B_3\td\alpha + B_2\td\beta
    \end{cases},
    \end{align*}
    which we will simply write as $\pr{\mathcal L-B}(\td\alpha,\td\beta)=0.$
    If $\td\alpha$ and $\td\beta$ are positive on $\bb R_+,$ then writing $a= \alpha/\tilde{\alpha}$, $b=\beta/\tilde{\beta}$, we have 
    \[
    \begin{split}
        \int_0^\infty &\left( \frac{1}{\omega}((\alpha')^2 + (\beta')^2) + B_1\alpha^2 + B_2 \beta^2 - 2B_3 \alpha\beta \right) \rho \,dr 
        \\& = \int_0^\infty \left( \frac{\tilde{\alpha}^2}{\omega} (a')^2 + \frac{\tilde{\beta}^2}{\omega} (b')^2 + B_3 \tilde{\alpha}\tilde{\beta} (a -b)^2 \right) \rho \,dr.
    \end{split}
    \]
\end{lemma}

\begin{proof}
    It is convenient to set $\sigma = {\rho}/{\omega}$. 
    The equation $\pr{\mathcal{L}-B}(\tilde{\alpha},\tilde{\beta}) = 0$ then becomes 
    \begin{equation}
    \label{eq:pos-eigenmodes-1}
    \begin{split}
        \frac{1}{\omega}\left( \tilde{\alpha}'' + \frac{\sigma'}{\sigma} \tilde{\alpha}' \right) &= B_1 \tilde{\alpha} -B_3 \tilde{\beta}
        \,\text{ and}\\
        \frac{1}{\omega} \left(\tilde{\beta}'' + \frac{\sigma'}{\sigma} \tilde{\beta}'\right) &= B_2 \tilde{\beta} -B_3 \tilde{\alpha}.
    \end{split}
    \end{equation}
    Differentiating $a=\alpha/\tilde{\alpha}$, first note that \[\tilde{\alpha}^2 (a')^2 - (\alpha')^2 = - 2 a \alpha' \tilde{\alpha}' + a^2 (\tilde{\alpha}')^2.\] 
    Using \eqref{eq:pos-eigenmodes-1}, we have 
    \[
    \begin{split}
        \left(\frac{\alpha^2}{\tilde{\alpha}}\tilde{\alpha}' \sigma\right)' &= \sigma\left(\frac{\alpha^2}{\tilde{\alpha}}\tilde{\alpha}'' + \frac{\sigma'}{\sigma}\frac{\alpha^2}{\tilde{\alpha}} \tilde{\alpha}' + 2a \alpha'\tilde{\alpha}' - a^2 (\tilde{\alpha}')^2 \right)
        = \rho\left(\tilde{\alpha}a^2 (B_1 \tilde{\alpha} - B_3 \tilde{\beta})\right) + \sigma\left(2a \alpha'\tilde{\alpha}' - a^2 (\tilde{\alpha}')^2\right).
    \end{split}
    \]
    As $\alpha$ and $\tilde{\alpha}$ are both even (resp. both odd), $(\alpha \tilde{\alpha}')(0)=(0,0)$. Integrating by parts gives
    \[\int_0^\infty \left((\alpha')^2-\tilde{\alpha}^2 (a')^2\right)\sigma\, dr= -\int_0^\infty \tilde{\alpha}a^2 (B_1 \tilde{\alpha} - B_3 \tilde{\beta})\rho\, dr.\]
    By the same argument (interchanging $\alpha,\tilde{\alpha},a$ for $\beta,\tilde{\beta},b$), 
    \[\int_0^\infty \left( (\beta')^2-\tilde{\beta}^2 (b')^2\right)\sigma\, dr= -\int_0^\infty \tilde{\beta}b^2 (B_2 \tilde{\beta} - B_3 \tilde{\alpha})\rho\, dr.\]
    Summing these and collecting $B$ terms on the right-hand side yields
    \[
    \begin{split}
    \int_0^\infty &\left( (\alpha')^2 + (\beta')^2) \right)\sigma \, dr 
    -\int_0^\infty \left(\tilde{\alpha}^2 (a')^2 + \tilde{\beta}^2 (b')^2\right)\sigma\,dr\\
    &=\int_0^\infty \left( -B_1\tilde{\alpha}^2 a^2 -B_2 \tilde{\beta}^2 b^2 + B_3 \tilde{\alpha}\tilde{\beta}(a^2+b^2) \right)\rho\,dr.
    \end{split}
    \]
    Move the $B_1,B_2$ terms to the left-hand side and subtract $\int_0^\infty 2B_3\alpha\beta \rho\,dr$ to both sides, recalling that $\alpha = a\tilde{\alpha}$, $\beta = b\tilde{\beta}$; this completes the proof. 
\end{proof}

\subsection{$m=3$ mode}
\label{sec:3-mode}

In this section, we prove the non-negativity of the form $Q_3,$ which is given in~\eqref{eq:Q-m} by
\begin{align*}
Q_3(\alpha,\beta)
= \int_0^\infty
\pr{
	\frac{(\alpha')^2+(\beta')^2}{\omega}
	+\left(-1+\frac{9}{2\omega}\right)\alpha^2
	+\left(
	-1+\frac{5}{\omega}
	+\frac{2}{\omega^2}
	\right)\beta^2
	-\frac{9\sqrt2\,\sinh(2r)}{\omega^2}\alpha\beta
}\rho\,dr
\end{align*}
for any odd function $\alpha$ and even function $\beta$ in $H^1_W(\bb R).$
Note this is the symmetry required by \eqref{eq:m-odd-parity}.
We will use the fact that $\pr{\tilde{\alpha},\tilde{\beta}}=\pr{\frac{3}{4}\sinh(2r), \frac{1}{\sqrt{2}}\cosh(2r)}$ is a $0$-eigenvector of $Q_3,$ generated by the rotation field $\frac{1}{2}\pr{V_{13}^\perp + V_{24}^\perp},$ satisfying the odd-even symmetry; see Table~\ref{tab:evectors}.

\begin{lemma}
    \label{lem:mode-3-form}
    Given any odd function $\alpha$ and even function $\beta$ in $H^1_W(\bb R),$ if we write 
    \begin{align*}
    \alpha(r) = \frac 34\sinh(2r)\,a(r)
    \,\,\text{ and }\,\,
    \beta(r) =  \frac 1{\sqrt 2}\cosh(2r)\,b(r),
    \end{align*}
    then
    \begin{align*}
    Q_3(\alpha,\beta)
    =& \int_0^\infty
    \pr{
    	\frac {9\sinh^2(2r)}{16\omega}(a')^2
    	+ \frac {\cosh^2(2r)}{2\omega}(b')^2
    	+ \frac{27\sinh^2(2r)\cosh(2r)}{8\omega^2}(a-b)^2
    }\rho\,dr.
    \end{align*}
\end{lemma}

\begin{cor}
    \label{cor:mode-3-analysis}
    Given any odd function $\alpha$ and even function $\beta$ in $H^1_W(\bb R),$ $Q_3(\alpha,\beta)\ge 0.$
    Moreover, $Q_3(\alpha,\beta)>0$ except when $(\alpha,\beta)=C(\frac{3}{4}\sinh(2r), \frac{1}{\sqrt{2}}\cosh(2r))$ for some constant $C\in\mathbb{R}$.
\end{cor}

\begin{proof}
	[Proof of Lemma~\ref{lem:mode-3-form}]
    Since $\pr{\tilde{\alpha},\tilde{\beta}}=\pr{\frac{3}{4}\sinh(2r), \frac{1}{\sqrt{2}}\cosh(2r)}$ is a $0$-eigenvector of the associated operator, $(\mathcal{L}-B)(\tilde{\alpha},\tilde{\beta})=(0,0)$, where $\mathcal{L}$, $B$ are as in Lemma \ref{lem:pos-eigenvector} with
    \begin{align*}
    B_1 &= -1 + \frac 9{2\omega},\\
    B_2 &= -1 + \frac 5\omega + \frac 2{\omega^2},\text{ and}\\
    B_3 &= \frac{9\sqrt 2 \sinh(2r)}{2\omega^2},
    \end{align*}
    the desired form follows from Lemma~\ref{lem:pos-eigenvector} as $\tilde{\alpha}$ and $\tilde{\beta}$ are positive on $\mathbb{R}_+.$ 
\end{proof}

\subsection{$m=4$ mode}
\label{sec:4-mode}

In this section, we prove the non-negativity of the form $Q_4,$ which is given in~\eqref{eq:Q-m} by
\begin{align*}
Q_4(\alpha,\beta)
= \int_0^\infty
\pr{
	\frac{(\alpha')^2+(\beta')^2}{\omega}
	+\left(-1+\frac{8}{\omega}\right)\alpha^2
	+\left(
	-1+\frac{17}{2\omega}
	+\frac{2}{\omega^2}
	\right)\beta^2
	-\frac{12\sqrt2\,\sinh(2r)}{\omega^2}\alpha\beta
}\rho\,dr
\end{align*}
for any even function $\alpha$ and odd function $\beta$ in $H^1_W(\bb R).$
Note this symmetry is that required by \eqref{eq:m-even-parity}.

\begin{lemma}
    \label{lem:mode-4-form}
    Given any even function $\alpha$ and odd function $\beta$ in $H^1_W(\bb R),$ if we write 
    \begin{align*}
    \alpha(r) = \cosh(2r)\, a(r)
    \,\,\text{ and }\,\,
    \beta(r) = \frac 1{\sqrt 2}\sinh(2r)\, b(r),
    \end{align*}
    then
    \begin{align*}
    Q_4(\alpha,\beta)
    = \int_0^\infty
    \Bigg(&\frac{\cosh^2(2r)}\omega (a')^2 
	+ \frac{\sinh^2(2r)}{2\omega} (b')^2
	+ \frac{6\cosh(2r)\sinh^2(2r)}{\omega^2} (a-b)^2\\
	& + \frac{\cosh(2r) (7\cosh(2r)-3)}{2\omega} a^2
	+ \frac{2\sinh^2(2r)}{\omega^2} b^2
    \Bigg)\rho\,dr.
    \end{align*}
\end{lemma}

\begin{cor}
    \label{cor:mode-4-analysis}
    Given any even function $\alpha$ and odd function $\beta$ in $H^1_W(\bb R),$ we have $Q_4(\alpha,\beta)\ge 0.$
    Moreover, $Q_4(\alpha,\beta)>0$ unless $(\alpha,\beta)\equiv 0$.
\end{cor}

\begin{proof}
	[Proof of Lemma~\ref{lem:mode-4-form}]
    The lemma follows from Lemma \ref{lem:pos-eigenvector} with the matrix entries defined by 
    \begin{align*}
    B_1 &= \left(-1 + \frac{8}{\omega}\right) - \frac{7\cosh(2r)-3}{2\omega \cosh(2r)},\\
    B_2 &= \left(-1+\frac{17}{2\omega}
	+\frac{2}{\omega^2}\right) - \frac{4}{\omega^2},\text{ and}\\
    B_3 &= \frac{6\sqrt2\,\sinh(2r)}{\omega^2}
    \end{align*}
    since straightforward calculations imply that $\pr{\tilde{\alpha},\tilde{\beta}}=\pr{\cosh(2r),\frac{1}{\sqrt{2}}\sinh(2r)}$ is a solution of $(\mathcal{L}-B)(\tilde{\alpha},\tilde{\beta})=(0,0)$. 
    Note that $(\tilde{\alpha},\tilde{\beta})$ is an even-odd pair, each of which is positive on $\mathbb{R}_+$.
\end{proof}

\section{\bf Modes $1$ and $2$}
\label{sec:12-mode}

We will keep using the convention in \eqref{convention-omega-rho}.

\subsection{$m=1$ mode}
\label{sec:1-mode}

In this section, we consider the form $Q_1$ given in~\eqref{eq:Q-m} by
\begin{align*}
Q_1(\alpha,\beta)
= \int_0^\infty
\pr{
	\frac{(\alpha')^2+(\beta')^2}{\omega}
	+\left(-1+\frac{1}{2\omega}\right)\alpha^2
	+\left(
	-1+\frac{1}{\omega}
	+\frac{2}{\omega^2}
	\right)\beta^2
	-\frac{3\sqrt2\,\sinh(2r)}{\omega^2}\alpha\beta
}\rho\,dr
\end{align*}
for odd functions $\alpha$ and even functions $\beta$ in $H^1_W(\bb R)$. 
We will show that $Q_1$ is nonnegative on odd-even pairs that are orthogonal to the $(-\frac{1}{2})$-eigenvector $(\sinh(r), \sqrt 2\cosh(r))$ (corresponding to translation generated by $e_3^\perp$).
Note that the symmetry constraint is required by \eqref{eq:m-odd-parity}.
Recall in this mode we also have the 0-eigenvector $\pr{\frac{3}{4}\sinh(2r), -\frac{\sqrt{2}}{2}}$ (corresponding to the rotation generated by $\frac{1}{2}(V_{13}^\perp - V_{24}^\perp)$).
See Table~\ref{tab:evectors}.
This non-negativity property is more technical than the cases in Section~\ref{sec:34-mode}, and we will first rewrite the form.

\begin{lemma}
\label{lem:mode-1-form}
    Given any odd function $\alpha$ and even function $\beta$ in $H^1_W(\bb R),$ if we write
    \begin{align*}
    h(r)=\frac{\sinh(r)}{\omega(r)}\alpha(r) + \frac{\sqrt{2}\cosh(r)}{\omega(r)}\beta(r) \,\,\text{ and }\,\,
    d(r) = \frac{\alpha(r)}{\sinh(r)}-\frac{\beta(r)}{\sqrt{2}\cosh(r)},
    \end{align*}
    then
    \begin{align}
    \label{eq:mode-1-rewritten}
    &Q_1(\alpha,\beta)\\\nonumber
    =& \int_0^\infty
    \pr{\frac{\omega^2}{\omega^2+4\omega+8}(h')^2-\frac\omega 2 h^2}\rho\,dr\\\nonumber
    &+ \int_0^\infty
    4\sinh^2(r)\cosh^2(r) \pr{\frac{4}{\omega^4} + \frac{1}{\omega^2(\omega+2)}}
    \pr{d - \frac{h'}{\frac{4\sinh(r)\cosh(r)}{\omega^2} + \frac{\sinh(r)\cosh(r)}{\omega+2}}}^2\rho\,dr\\\nonumber
    &+ \int_0^\infty
    \frac{2(\omega+2)^2\sinh^2(r)}{\omega^2} \pr{\pr{\frac{2\cosh(r)}{\omega+2} d}'}^2\rho\,dr.
    \end{align}
\end{lemma}

\begin{proof}
    In this proof, we use the convention that $c(r)=\cosh(r)$ and $s(r)=\sinh(r)$ so
    \begin{align*}
    Q_1(\alpha,\beta)
    = \int_0^\infty
    \pr{
    	\frac{(\alpha')^2+(\beta')^2}{\omega}
    	+\left(-1+\frac{1}{2\omega}\right)\alpha^2
    	+\left(
    	-1+\frac{1}{\omega}
    	+\frac{2}{\omega^2}
    	\right)\beta^2
    	-\frac{6\sqrt2 sc}{\omega^2}\alpha\beta
    }\rho\,dr.
    \end{align*}
    The first step is similar to the proof of Lemmas~\ref{lem:mode-3-form} and~\ref{lem:mode-4-form}.
    As the translational mode $\pr{\tilde{\alpha},\tilde{\beta}} = \pr{s, \sqrt{2}c}$ is a $(-1/2)$-eigenvector of $Q_1$, and hence the associated operator satisfies $(\mathcal{L}-B)(\tilde{\alpha},\tilde{\beta})=(0,0)$, where $\mathcal{L}$, $B$ are as in Lemma \ref{lem:pos-eigenvector} with $B_1 = \left(-1+\frac{1}{2\omega}\right) + \frac{1}{2}$, $B_2 = \left(
    	-1+\frac{1}{\omega}
    	+\frac{2}{\omega^2}
    	\right) + \frac{1}{2}$, $B_3 = \frac{3\sqrt{2}sc}{\omega^2}$. 
    As $\tilde{\alpha}$ and $\tilde{\beta}$ are positive on $\mathbb{R}_+$, Lemma \ref{lem:pos-eigenvector} gives
	\begin{align*}
        Q_1(\alpha,\beta)
		=\,&
		\int_0^\infty
		\pr{
			\frac{s^2}{\omega}(a')^2
			+\frac{2c^2}{\omega}(b')^2
			+\frac{6s^2c^2}{\omega^2}(a-b)^2
		}\rho\,dr
		-\int_0^\infty
		\pr{
			\frac12s^2a^2+c^2b^2
		}\rho\,dr.
	\end{align*}

	Next, we rewrite $Q_1$ using the variance formula, which says  
	\begin{align}\label{eq:variance}
	kx^2+\ell y^2=k\ell(x-y)^2+(kx+\ell y)^2
	\end{align}
	for $k+\ell=1.$
	We apply the formula to both $h = \frac{s^2}{\omega}a + \frac{2c^2}{\omega}b$ and its derivative
	\begin{align*}
	h' - \frac{4sc}{\omega^2} (a-b) = \frac{s^2}{\omega}a' + \frac{2c^2}{\omega}b'
	\end{align*} 
	to obtain 
	\begin{align*}
	\frac{s^2}{\omega}a^2 + \frac{2c^2}{\omega}b^2
	&= \frac{2s^2c^2}{\omega^2}(a-b)^2 + h^2
	\,\text{ and}\\
	\frac{s^2}{\omega}(a')^2
	+\frac{2c^2}{\omega}(b')^2
	&= \frac{2s^2c^2}{\omega^2}(a'-b')^2 + \pr{h' - \frac{4sc}{\omega^2}d}^2
	\end{align*}
	since $\frac{s^2}{\omega} + \frac{2c^2}{\omega}=1.$ 
	Considering $d:=a-b,$ we derive 
	\begin{align*}
	Q_1(\alpha,\beta)
	=\,&
	\int_0^\infty
	\pr{
		\frac{2s^2c^2}{\omega^2}(d')^2 + \pr{h' - \frac{4sc}{\omega^2}d}^2
		+\frac{6s^2c^2}{\omega^2}d^2
	}\rho\,dr
	-\int_0^\infty
	\pr{
		\frac{s^2c^2}{\omega}d^2 + \frac\omega 2 h^2
	}\rho\,dr\\
	=\,& 
	\int_0^\infty
	\pr{
		\pr{h' - \frac{4sc}{\omega^2}d}^2
		- \frac\omega 2h^2
		+ \frac{2s^2c^2}{\omega^2}(d')^2
		+ \frac{s^2c^2(6-\omega)}{\omega^2}d^2
	}\rho\,dr.
	\end{align*}

    For further simplification, we consider
    \begin{align}\label{d0-m=1}
    d_0:=\frac 32 c - \pr{-\frac 1{2c}}
    = \frac{3c^2+1}{2c}
    = \frac{\omega+2}{2c},
    \end{align}
    which is the difference $d$ when $(\alpha,\beta)$ is given by the rotational Jacobi field in the $m=1$ mode.
    Iterate integrations by parts in a direct way implies
    \begin{align*}
    \int_0^\infty
    \pr{\frac{2s^2c^2}{\omega^2}\pr{d'}^2
    + \frac{s^2c^2(6-\omega)}{\omega^2}d^2}\rho\,dr
     &= \int_0^\infty
    \pr{\frac{2s^2c^2}{\omega^2}d_0^2 \pr{\pr{\frac d{d_0}}'}^2
    + \frac{4s^2c^2}{\omega^2(\omega+2)} d^2} \rho\,dr,
    \end{align*}
    and this allows us to rewrite
    \begin{align*}
    Q_1(\alpha,\beta)
     =& \int_0^\infty
     \pr{
     	\pr{h'-\frac{4sc}{\omega^2}d}^2
     	- \frac\omega{2}h^2
     	+ \frac{4s^2c^2}{\omega^2(\omega+2)} d^2
     + \frac{2s^2c^2}{\omega^2}d_0^2 \pr{\pr{\frac d{d_0}}'}^2
     }
     \rho\,dr.
     \end{align*}
     Completing the square for $h'$ and $d$ leads to
     \begin{align*}
     \pr{h'-\frac{4sc}{\omega^2}d}^2 + \frac{4s^2c^2}{\omega^2(\omega+2)} d^2
     =&\,
     \pr{\frac{16s^2c^2}{\omega^4} + \frac{4s^2c^2}{\omega^2(\omega+2)}}d^2
     - \frac{8sc}{\omega^2}h'd
     + (h')^2\\
     =&\, \pr{\frac{16s^2c^2}{\omega^4} + \frac{4s^2c^2}{\omega^2(\omega+2)}}
     \pr{d- \pr{\frac{16s^2c^2}{\omega^4} + \frac{4s^2c^2}{\omega^2(\omega+2)}}^{-1}\frac{4sc}{\omega^2}h'}^2\\
     & + (h')^2 
     - \pr{\frac{16s^2c^2}{\omega^4} + \frac{4s^2c^2}{\omega^2(\omega+2)}}^{-1}\frac{16s^2c^2}{\omega^4}(h')^2.
     \end{align*}
    Thus, we finally obtain
    \begin{align*}
    Q_1(\alpha,\beta)
    =& \int_0^\infty
    \pr{F(h')^2-\frac\omega 2 h^2}\rho\,dr\\
    &+ \int_0^\infty
    \pr{\frac{16s^2c^2}{\omega^4} + \frac{4s^2c^2}{\omega^2(\omega+2)}}
    \pr{d- \pr{\frac{16s^2c^2}{\omega^4} + \frac{4s^2c^2}{\omega^2(\omega+2)}}^{-1}\frac{4sc}{\omega^2}h'}^2\rho\,dr\\
    &+ \int_0^\infty
    \frac{2s^2c^2}{\omega^2}d_0^2 \pr{\pr{\frac d{d_0}}'}^2\rho\,dr\,
    \end{align*}
    with 
    \begin{align*}
    F 
    = 1 - \pr{\frac{16s^2c^2}{\omega^4} + \frac{4s^2c^2}{\omega^2(\omega+2)}}^{-1}\frac{16s^2c^2}{\omega^4} 
    =\frac{\omega^2}{\omega^2+4\omega+8}.
    \end{align*}
    This finishes the proof of the lemma.
\end{proof}

The non-negativity of $Q_1$ subject to the orthogonality condition is then reduced to the following Sturm--Liouville type problem.

\begin{lemma}
\label{lem:mode-1-SL}
    Let $f$ be an even function in $H^1_W(\bb R).$ 
    If $\int_0^\infty 
    f \omega 
    \rho \,dr=0,$
    then 
    \begin{align*}
    \frac 12 \int_0^\infty
    f^2 \omega\rho\,dr
    \le \int_0^\infty (f')^2
    \cdot \frac{\omega^2\rho}{\omega^2+4\omega+8} \,dr.
    \end{align*}
\end{lemma}

\begin{proof}
	We use the convention that $c(r)=\cosh(r)$ and $s(r)=\sinh(r).$
	Consider
	\begin{align*}
	h_1:=\frac{s^2}\omega\cdot \frac 32c + \frac{2c^2}\omega\cdot \frac{-1}{2c}
	= \frac{c}{2\omega}\pr{3s^2-2},
	\end{align*}
	which is the function $h$ when $(\alpha,\beta)$ is given by the rotational Jacobi field in the $m=1$ mode.
	Then direct calculations lead to 
	\begin{align}\label{eq:h1-eq}
	h_1'
	&= \frac 1{2\omega}
	\pr{3s^3+6sc^2-2s}
	+ \frac{-1}{2\omega^2} 6sc\pr{3cs^2-2c}
	= \frac s{2\omega^2} \pr{\omega^2+4\omega+8}.
	\end{align}
	This implies
	\begin{align*}
	\frac{\omega^2}{\omega^2+4\omega+8} h_1'
	&= \frac{\omega^2}{\omega^2+4\omega+8}\cdot \frac s{2\omega^2} \pr{\omega^2+4\omega+8}
	= \frac s2.
	\end{align*}
	Thus, if we consider the operator $L_1$ defined by
	\begin{align*}
	L_1 f:= \frac{1}{\rho\omega} \pr{2\rho  \frac{\omega^2}{\omega^2+4\omega+8} f'}',
	\end{align*}
	then it follows that
	\begin{align*}
	L_1h_1
	= \frac 1{\rho\omega}
	(\rho s)
	= \frac 1{\rho\omega} \pr{-\rho \omega h_1}
	= -h_1.
	\end{align*}
	Since $h_1$ has exactly one zero on the positive real line, it follows that it is the second even eigenfunction of $L_1$ with the second eigenvalue $1.$
	As a result, since the function $f$ is orthogonal to the first eigenfunction $1$ with respect to the weight $\omega\rho\,dr,$ it follows that
	\begin{align*}
	-\int_0^\infty (fL_1f) \rho\omega\,dr
	\ge \int_0^\infty f^2 \rho\omega\,dr.
	\end{align*}
	Integration by parts implies the desired inequality based on the definition of $L_1.$
\end{proof}

Combining Lemmas~\ref{lem:mode-1-form} and~\ref{lem:mode-1-SL} gives us the desired estimate when $m=1.$
We summarize it in the following proposition.

\begin{prop}
    \label{prop:mode-1-analysis}
    Let $\alpha$ be an odd function and $\beta$ be an even function in $H^1_W(\bb R)$ such that
    \begin{align}\label{eq:m=1-assumption}
        \int_0^\infty
        \pr{\sinh(r)\alpha(r) + \sqrt 2\cosh(r)\beta(r)}
        \rho\,dr=0.
    \end{align}
    Then $Q_1(\alpha,\beta)\ge 0,$ with equality only if $(\alpha,\beta)=C\pr{\frac{3}{4}\sinh(2r),-\frac{\sqrt{2}}{2}}$ for some constant $C\in\mathbb{R}$. 
\end{prop}

\begin{proof}
    As in Lemma~\ref{lem:mode-1-form}, let $h(r):=\frac{\sinh(r)}{\omega(r)}\alpha(r) + \frac{\sqrt 2\cosh(r)}{\omega(r)}\beta(r)$. The assumption~\eqref{eq:m=1-assumption} means $\int_0^\infty h \omega\rho\,dr=0,$
    so Lemma~\ref{lem:mode-1-SL} implies that the first term in the formula \eqref{eq:mode-1-rewritten} for $Q_1$ is nonnegative. The remaining terms are all square terms, and so $Q_1(\alpha,\beta)\geq 0$. 

    Moreover, if $Q_1(\alpha,\beta) =0$, then all the square terms must vanish. 
    In particular, writing $s(r)=\sinh(r)$ and $c(r)=\cosh(r)$ as in the proof of Lemma \ref{lem:mode-1-form}, we have 
    \begin{equation}
    \label{eq:mode-1-d0}
        \frac{\alpha(r)}{\sinh(r)}-\frac{\beta(r)}{\sqrt{2}\cosh(r)} = C\cdot \frac{\omega+2}{2c}
    \end{equation} 
    for some $C\in \mathbb{R}$, so
    \[h' = \left(\frac{4}{\omega^2}+\frac{1}{\omega+2}\right)sc \left(C \cdot \frac{\omega+2}{2c} \right)= C\cdot \frac{s}{2} \cdot \frac{\omega^2+4\omega + 8}{\omega^2}. \]
    Recall from the proof of Lemma \ref{lem:mode-1-SL} that $h_1 = \frac{c}{2\omega}(3s^2-1)$ satisfies \eqref{eq:h1-eq}.
    Since $\int_0^\infty h \omega\rho\,dr=\int_0^\infty h_1 \omega\rho\,dr=0$, it follows that indeed 
    \begin{equation}
    \label{eq:mode-1-h1}
        \frac{\sinh(r)}{\omega(r)}\alpha(r) + \frac{\sqrt 2\cosh(r)}{\omega(r)}\beta(r) = Ch_1
    \end{equation} Simple linear algebra applied to \eqref{eq:mode-1-d0} and \eqref{eq:mode-1-h1} yields that $(\alpha,\beta) = C\pr{\frac{3}{4}\sinh(2r), -\frac{\sqrt{2}}{2}}$. 
\end{proof}

\subsection{$m=2$ mode}
\label{sec:2-mode}

In this section, we consider the form $Q_2$  given in~\eqref{eq:Q-m} by
\begin{align*}
Q_2(\alpha,\beta)
= \int_0^\infty
\pr{
	\frac{(\alpha')^2+(\beta')^2}{\omega}
	+\left(-1+\frac{2}{\omega}\right)\alpha^2
	+\left(
	-1+\frac{5}{2\omega}
	+\frac{2}{\omega^2}
	\right)\beta^2
	-\frac{6\sqrt2\,\sinh(2r)}{\omega^2}\alpha\beta
}\rho\,dr
\end{align*}
for even functions $\alpha$ and odd functions $\beta$ in $H^1_W(\bb R)$. We will show that $Q_2$ is nonnegative on even-odd pairs that are orthogonal to the $(-\frac{1}{2})$-eigenvector $(\sqrt 2\cosh(r), \sinh(r)),$ corresponding to translation generated by $-e_1^\perp$; see Table~\ref{tab:evectors}.
Note that the parity constraint is required by \eqref{eq:m-even-parity}.
As in the $m=1$ case, this non-negativity property is more technical than the cases in Section~\ref{sec:34-mode}, and we will first rewrite the form. 

\begin{lemma}
    \label{lem:mode-2-form}
    Given any even function $\alpha$ odd even function $\beta$ in $H^1_W(\bb R),$ if we write
    \begin{align*}
    h(r)=\frac{\sqrt{2}\cosh(r)}{\omega(r)}\alpha(r) + \frac{\sinh(r)}{\omega(r)}\beta(r)\,\,\text{ and }\,\,
    d(r)= \frac{\alpha(r)}{\sqrt{2}\cosh(r)} - \frac{\beta(r)}{\sinh(r)},
    \end{align*}
    then 
    \begin{align}
    \label{eq:mode-2-rewritten}
    Q_2(\alpha,\beta)
    =&\, \frac 12 \int_0^\infty
    \pr{(h')^2-\omega h^2}\rho\,dr
    + \frac 12 \int_0^\infty
    \pr{h' + \frac{8\sinh(r) \cosh(r)}{\omega^2}d}^2\rho\,dr\\\nonumber
    &+ \int_0^\infty
    \frac{2\cosh^4(r) \sinh^2(r)}{\omega^2}\pr{\pr{\frac{d}{\cosh(r)}}'}^2\rho\,dr\\\nonumber
    &+ \int_0^\infty
    \frac{4\sinh^2(r)}{\omega^4}
    \pr{18\sinh^6r+39\sinh^4r+20\sinh^2(r)} d^2\rho\,dr.
    \end{align} 
\end{lemma}

\begin{proof}
	In this proof, we use the convention that $s(r)=\sinh(r)$ and $c(r)=\cosh(r)$ so    
    \begin{align*}
    Q_2(\alpha,\beta)
    = \int_0^\infty
    \pr{
    	\frac{(\alpha')^2+(\beta')^2}{\omega}
    	+\left(-1+\frac{2}{\omega}\right)\alpha^2
    	+\left(
    	-1+\frac{5}{2\omega}
    	+\frac{2}{\omega^2}
    	\right)\beta^2
    	-\frac{12\sqrt2 sc}{\omega^2}\alpha\beta
    }\rho\,dr
    \end{align*}
    As the translational mode $\pr{\tilde{\alpha},\tilde{\beta}} = \pr{\sqrt{2} c, s}$ is a $(-1/2)$-eigenvector of $Q_2$, and hence the associated operator satisfies $(\mathcal{L}-B)(\tilde{\alpha},\tilde{\beta})=(0,0)$, where $\mathcal{L}$, $B$ are as in Lemma \ref{lem:pos-eigenvector} with $B_1 = \left(-1+\frac{2}{\omega}\right)+\frac{1}{2}$, $B_2 = \left(-1+\frac{5}{2\omega}
    	+\frac{2}{\omega^2}\right)+\frac{1}{2}$, $B_3 = \frac{12\sqrt2 sc}{\omega^2}$. 
    As $\tilde{\alpha}$ and $\tilde{\beta}$ are positive on $\mathbb{R}_+$, Lemma \ref{lem:pos-eigenvector} gives
	\begin{align*}
	Q_2(\alpha,\beta)
	=\,&
	\int_0^\infty
	\pr{
		\frac{2c^2}{\omega}(a')^2
		+\frac{s^2}{\omega}(b')^2
		+\frac{12s^2c^2}{\omega^2}a^2
		+\frac{12s^2c^2}{\omega^2}b^2
		-\frac{24s^2c^2}{\omega^2}ab
		-c^2a^2-\frac12s^2b^2
	}\rho\,dr\\
	=\,&
	\int_0^\infty
	\pr{
		\frac{2c^2}{\omega}(a')^2
		+\frac{s^2}{\omega}(b')^2
		+\frac{12s^2c^2}{\omega^2}(a-b)^2
	}\rho\,dr
	-
	\int_0^\infty
	\pr{
		c^2a^2+\frac12s^2b^2
	}\rho\,dr.
	\end{align*}

    Next, we rewrite $Q_2$ using the variance formula \eqref{eq:variance}.
    Using $\frac{s^2}{\omega} + \frac{2c^2}{\omega}=1$ and the same idea as in the proof of Lemma~\ref{lem:mode-1-form} and considering $d:=a-b,$ we obtain 
    \begin{align*}
    Q_2(\alpha,\beta)
    =& \int_0^\infty 
    \pr{\frac{2c^2s^2}{\omega^2}(a'-b')^2
    	+ \pr{h' + \frac{4sc}{\omega^2}(a-b)}^2
    	+ \frac{12s^2c^2}{\omega^2} (a-b)^2
    }\rho\,dr\\
    &- \int_0^\infty 
    \pr{\frac{c^2s^2}\omega(a-b)^2 + \frac\omega 2h^2}\rho\,dr\\
    =& \int_0^\infty
    \pr{\frac{2c^2s^2}{\omega^2}d' 
    	+ \pr{h' + \frac{4sc}{\omega^2}d}^2
    	+ \frac{12s^2c^2}{\omega^2}d^2
    }\rho\,dr
    - \int_0^\infty
    \pr{\frac{c^2s^2}\omega d^2 + \frac\omega 2h^2}\rho\,dr.
    \end{align*}
    Expanding the squared term by writing  
    \begin{align*}
    \pr{h' + \frac{4sc}{\omega^2}d}^2
    = \frac 12(h')^2
    + \frac 12 \pr{h' + \frac{8sc}{\omega^2}d}^2
    - \frac{16s^2c^2}{\omega^4}d^2,
    \end{align*}
    we get
    \begin{align*}
    Q_2(\alpha,\beta)
    =& \frac 12 \int_0^\infty
    \pr{(h')^2-\omega h^2}\rho\,dr
    + \frac 12 \int_0^\infty
    \pr{h' + \frac{8sc}{\omega^2}d}^2\rho\,dr\\
    &+ \int_0^\infty
    \pr{\frac{2c^2s^2}{\omega^2}(d')^2 
    	- \frac{16s^2c^2}{\omega^4}d^2
    	+ \frac{12s^2c^2}{\omega^2}d^2
    	- \frac{c^2s^2}\omega d^2
    }\rho\,dr.
    \end{align*}
    Iterate integrations by parts in a direct way implies
    \begin{equation}\label{m-2-sug}
    \begin{split}
    \int_0^\infty
    \frac{2c^2s^2}{\omega^2}(d')^2\rho\,dr
    = & \int_0^\infty
    \pr{\frac{2c^4s^2}{\omega^2}\pr{\pr{\frac dc}'}^2
    	+ \frac{2s^4}{\omega^2} d^2
    }\rho\,dr\\
    &+ \int_0^\infty \frac{d^2}{c^2}
    \pr{
    	\frac{2c^3s^3}{\omega^2}\cdot \frac 32sc
    	- \frac{6c^2s^4+6c^4s^2}{\omega^2}
    	+ \frac{4c^3s^3}{\omega^3}\cdot 6sc
    }\rho\,dr.
    \end{split}
    \end{equation}
    Thus, the coefficients of $d^2$ combine to
    \begin{align*}
    &\, \frac{2s^4}{\omega^2}
    + \frac 1{c^2} \pr{
    	\frac{2c^3s^3}{\omega^2}\cdot \frac 32sc
    	- \frac{6c^2s^4+6c^4s^2}{\omega^2}
    	+ \frac{4c^3s^3}{\omega^3}\cdot 6sc
    }
    - \frac{16s^2c^2}{\omega^4}
    + \frac{12s^2c^2}{\omega^2}
    - \frac{c^2s^2}\omega \\
    =&\, \frac{4s^2}{\omega^4}
    \pr{18s^6+39s^4+20s^2}.
    \end{align*}
    Thus, we finally obtain 
    \begin{align*}
    Q_2(\alpha,\beta)
    =&\, \frac 12 \int_0^\infty
    \pr{(h')^2-\omega h^2}\rho\,dr
    + \frac 12 \int_0^\infty
    \pr{h' + \frac{8sc}{\omega^2}d}^2\rho\,dr\\
    &+ \int_0^\infty
    \frac{2c^4s^2}{\omega^2}\pr{\pr{\frac dc}'}^2\rho\,dr
    + \int_0^\infty
    \frac{4s^2}{\omega^4}
    \pr{18s^6+39s^4+20s^2} d^2\rho\,dr.
    \end{align*}
    This finishes the proof of the lemma.
\end{proof}

The non-negativity of $Q_2$ subject to the orthogonality condition is then reduced to the following Sturm--Liouville type problem.

\begin{lemma}
\label{lem:mode-2-SL}
    Let $f$ be an even function in $H^1_W(\bb R).$ 
    If $\int_0^\infty 
    f\omega 
    \rho \,dr=0,$
    then 
    \begin{align*}
    \int_0^\infty
    f^2 \omega\rho\,dr
    \le \int_0^\infty (f')^2 \omega\rho \,dr.
    \end{align*}
\end{lemma}

\begin{proof}
	Consider the function
	\begin{align*}
	h_2:=\cosh(2r)-3
	\end{align*}
	and the operator $L_2$ defined by
	\begin{align*}
	L_2 f:= \frac{1}{\rho\omega} \pr{\rho f'}'.
	\end{align*}
	Then straightforward calculations imply
	\begin{align*}
	L_2h_2
	= \frac{1}{\omega} \pr{4\cosh(2r) - \frac 32 \sinh^2(2r)}
    = -h_2.
	\end{align*}
	Since $h_2$ has exactly one zero on the positive real line, it follows that it is the second even eigenfunction of $L_2$ with the second eigenvalue $1.$ 
	As a result, since the function $f$ is orthogonal to the first eigenfunction $1$ with respect to the weight $\omega\rho\,dr,$ it follows that
	\begin{align*}
	-\int_0^\infty (fL_2f) \rho\omega\,dr
	\ge \int_0^\infty f^2 \rho\omega\,dr.
	\end{align*}
	Integration by parts implies the desired inequality based on the definition of $L_2.$
\end{proof}

Combining Lemmas~\ref{lem:mode-1-form} and~\ref{lem:mode-1-SL} gives us the desired estimate when $m=2.$
We summarize it in the following proposition.

\begin{prop}
    \label{prop:mode-2-analysis}
    Let $\alpha$ be an even function and $\beta$ be an odd function in $H^1_W(\bb R)$ such that
    \begin{align}\label{eq:m=2-assumption}
        \int_0^\infty
        \pr{\sqrt 2\cosh(r)\alpha(r) + \sinh(r)\beta(r)}
        \rho\,dr=0,
    \end{align}
    Then $Q_2(\alpha,\beta)\ge 0,$ with equality only if $(\alpha,\beta)=(0,0)$. 
\end{prop}

\begin{proof}
    As in Lemma~\ref{lem:mode-2-form}, let $h(r):=\frac{\sqrt 2\cosh(r)}{\omega(r)}\alpha(r) + \frac{\sinh(r)}{\omega(r)}\beta(r)$. The assumption~\eqref{eq:m=2-assumption} means $\int_0^\infty h \omega\rho\,dr=0,$ so Lemma~\ref{lem:mode-2-SL} implies that the first term in the formula \eqref{eq:mode-2-rewritten} for $Q_2$ is nonnegative. The remaining terms are all square terms, and so $Q_2((\alpha,\beta))\geq 0$. 

    Moreover, if $Q_2(\alpha,\beta)=(0,0),$ then the square terms must vanish, namely $d\equiv 0$, so $h'\equiv 0$. 
    Then $\int_0^\infty h \omega\rho\,dr=0$ implies that $h\equiv 0$, so $(\alpha,\beta)=(0,0)$. 
\end{proof}

\section{\bf Mode $m=0$}
\label{sec:0-mode}

We will keep using the convention in \eqref{convention-omega-rho}.

\subsection{$\theta$ mode}
\label{sec:mode-theta}

In this section, we consider the form $Q^\theta_0$, which is given in Proposition~\ref{prop:form-fourier} by 
\[Q^\theta_{0}(\alpha)= \int_0^\infty \left( \omega^{-1} (\alpha')^2 - \alpha^2 \right) \rho dr \]
for even functions $\alpha \in H^1_W(\mathbb{R})$. We will show the non-negativity of $Q^\theta_0$ for even $\alpha$ that are orthogonal to the $(-1)$-eigenfunction $\phi_0=1$ corresponding to dilation; see Table~\ref{tab:evectors}.   

\begin{prop}
    \label{prop:mode-0-theta-analysis}
    Let $\alpha$ be an even function in $H^1_W(\mathbb{R})$ and suppose that $\int_0^\infty \alpha \rho\, dr=0$. Then 
    \[Q^\theta_0(\alpha)\geq \int_0^\infty \frac{1}{\omega(\omega+1)} (\alpha')^2.\] 
    In particular, $Q^\theta_0(\alpha)\geq 0$ with equality only if $\alpha\equiv 0$.
\end{prop}

\begin{proof}
For convenience, we will write $\phi=\frac{3\cosh(2r)-1}{8}$ so that $\rho=e^{-\phi}$. 
Note that $\omega = \frac{3\cosh(2r)+1}{2} = \phi''+\frac{1}{2}$. Denote the Sturm-Liouville operator $L^\theta_0 u=\frac{1}{\rho}(\rho u')' = u'' - \phi' u'$ (which is self-adjoint on $L^2_\rho(\mathbb{R})$).
The proof is inspired by Helffer's proof of the Brascamp-Lieb inequality. 

Consider odd functions $h \in H^1_\rho(\mathbb{R})$ and let $f = h' - h\phi'$, so that $f' = h'' - h'\phi' - h\phi'' = L^\theta_0 h' - h \phi''$. We claim that 
\begin{equation}
    \label{eq:h-weighted-L2}
\int_\mathbb{R} h^2(\phi'' + \kappa) \rho \leq \int_{\mathbb{R}} ((h')^2 + h^2 \phi'')\rho \leq \int_\mathbb{R} \frac{(f')^2}{\phi''+\kappa}\rho,
\end{equation}
so long as the last integral is finite. 
By approximation, it is enough to prove the claim for smooth $h$ of compact support.

Let \[I:= \int_{\mathbb{R}} ((h')^2 + h^2 \phi'')\rho = \int_\mathbb{R}  (-hL^\theta_0 h' + h^2 \phi'')\rho=  -\int_\mathbb{R} f' h\rho .\]
In fact, as $h$ is odd, we certainly have $\int_\mathbb{R} h\rho=0$, so the Brascamp-Lieb inequality (\cite{BL76}; see also \cite[Theorem 2.1]{He98}) gives 
\[\int_\mathbb{R} \frac{(h')^2}{\phi''}\rho \geq  \int_\mathbb{R} h^2 \rho.\]
Then since $\phi'' = \frac{3}{2}\cosh(2r) \geq \frac{3}{2}$, we see that\footnote{One could also use a Poincar\'{e} inequality for the weighted space $(\mathbb{R},e^{-\phi})$.} 
\begin{equation}
    \label{eq:mode-0-1}
    I\geq \int_\mathbb{R}  h^2(\kappa + \phi'')\rho
\end{equation}
holds with $\kappa :=3/2$. On the other hand, using Cauchy-Schwarz we have 
\[
I =-\int_\mathbb{R} f' h\rho \leq \left(\int_\mathbb{R} \frac{(f')^2}{\phi''+\kappa} \rho \right)^{1/2}\left( \int_\mathbb{R}h^2(\phi''+\kappa) \rho\right)^{1/2} \leq \left(\int_\mathbb{R} \frac{(f')^2}{\phi''+\kappa} \rho \right)^{1/2} I^{1/2}.
\]
Dividing through by $I^{1/2}$ implies the claim. 

To proceed, we solve $L^\theta_0 u=\alpha$; as $\alpha\in L^2_\rho(\mathbb{R})$ is even, the solution $u\in H^2_\rho(\mathbb{R})$ is necessarily even. 
Define $h=u' \in H^1_\rho(\mathbb{R})$, so that $\alpha = \frac{1}{\rho}(h\rho)' = h'-h\phi$. By assumption $\alpha' \in L^2_{\rho/\omega}(\mathbb{R})$, so we may apply the claim with $f=\alpha$; using \eqref{eq:h-weighted-L2} and recalling $\omega = \phi''+\frac{1}{2}$, it follows that $\alpha' h \in L^1_\rho(\mathbb{R})$. Since $\alpha\in L^2_\rho(\mathbb{R})$, $h\in H^1_\rho(\mathbb{R})$, we may thereby justify the following integrations by parts:
\[\begin{split} 
I&:= \int_\mathbb{R}  ((h')^2 + h^2 \phi'')\rho=\int_\mathbb{R}  (-hL^\theta_0 h' + h^2 \phi'')\rho  =  -\int_\mathbb{R} \alpha' h\rho 
= \int_\mathbb{R} \alpha(h\rho)' = \int_\mathbb{R} \alpha^2\rho.
\end{split}\]
Again as $\omega= \phi''+\frac{1}{2}$, we conclude from \eqref{eq:h-weighted-L2} with $f=\alpha$ that \[I= \int_\mathbb{R} \alpha^2\rho \leq \int_\mathbb{R} \frac{(\alpha')^2}{\phi''+\kappa} = \int_\mathbb{R} \frac{1}{\omega+1} (\alpha')^2.\]
Using this inequality in the definition of $Q^\theta_0$ completes the proof.
\end{proof}

\subsection{$r$ mode}

In this section, we prove the non-negativity of the form $Q^r_0$, which is given in Proposition~\ref{prop:form-fourier} by 
\[Q^r_0(\beta)= \int_0^\infty \left( \omega^{-1} (\beta')^2 + \left(-1 +\frac{1}{2}\omega^{-1} + 2\omega^{-2}\right)\beta^2\right) \rho dr,\]
for odd functions $\beta \in H^1_W(\mathbb{R})$. Recall that in this mode we have the 0-eigenfunction $\psi_0=\sinh(2r)$ corresponding to the rotation generated by $V_{34}^\perp$; see Table~\ref{tab:evectors}.  

\begin{prop}
    \label{prop:mode-0-r-analysis}
    Let $\beta$ be an odd function in $H^1_W(\mathbb{R})$. Then $Q^r_0(\beta)\geq 0$, with equality if and only if $\beta = C\sinh(2r)$ for some constant $C\in\mathbb{R}$. 
\end{prop}
\begin{proof}

Consider the associated Sturm-Liouville operator 
\[L^r_0 f:= \frac{1}{\rho} \left( \frac{\rho}{\omega} f' \right)' - \left(-1 +\frac{1}{2}\omega^{-1} + 2\omega^{-2}\right)f\]
so that $Q^r_0(\beta) = -\int_0^\infty (\beta L^r_0\beta)\rho\,dr$.
Note $\psi_0 = \sinh(2r)$ is a 0-eigenfunction of $Q^r_0$, hence of the associated operator, $L^r_0 \psi_0=0$.  
Because $\sinh$ is positive on the positive real line, $\psi_0$ must be the first odd eigenfunction of $L^r_0$. 
In particular, all the remaining eigenvalues must be strictly positive, and the proposition follows. 
\end{proof}

\section{\bf Stability of the M\"{o}bius shrinker}
\label{sec:final}

This section contains the proof of the main theorem. 
It uses analyses at each Fourier mode detailed in the previous sections.
As a corollary of the rigidity cases in the mode analysis, a complete list of weakly unstable modes also follows.
We end with a few remarks.

\subsection{Higher modes}
The final ingredient of the proof of the main theorem is an estimate for higher modes.
It follows directly from the nature of the quadratic form defined in \eqref{eq:Q-m}.

\begin{prop}
    \label{prop:form-mono}
    Let $\alpha,\beta\in H^1_W(\bb R).$
    For any $m\geq 3$ we have $(Q_m - Q_{m-1})((\alpha,\beta))\geq 0$, with equality if and only if $(\alpha,\beta)=(0,0)$. 
\end{prop}

\begin{proof}
From the definition \eqref{eq:Q-m}, it follows that
\[
\begin{split}
(Q_{m} - Q_{m-1})((\alpha,\beta)) &= \int_0^\infty \left(\frac{2m-1}{2\omega} (\alpha^2+\beta^2) - \frac{3\sinh(2r)}{2\omega^2} ( 2\sqrt{2}\alpha\beta)\right) \rho dr
\\&\geq \int_0^\infty \left(\frac{2m-1}{2\omega}  - \frac{3\sqrt 2 \sinh(2r)}{2\omega^2} \right) \left(\alpha^2+\beta^2\right) \rho dr
\\&\geq \int_0^\infty \left(\frac{2m-4}{2\omega}   \right) \left(\alpha^2+\beta^2\right) \rho dr,
.\end{split}\]
where we estimate $2\alpha\beta \leq \alpha^2+\beta^2$ and note that $\frac{3\sqrt{2}\sinh(2r)}{\omega}<3$. 
The proposition follows as the coefficient on the right-hand side is strictly positive when $m\geq 3$. 
\end{proof}

\subsection{Proof of the main theorem}

We combine the analysis in Section~\ref{sec:34-mode}, \ref{sec:12-mode}, and \ref{sec:0-mode} with Proposition~\ref{prop:form-mono}.
By the decomposition given in Proposition~\ref{prop:form-fourier}, the main theorem follows.
As a corollary of the analysis, we also derive the precise list of all unstable eigenfunctions and Jacobi fields.

\begin{proof}
[Proof of Theorem~\ref{thm:main}]
    As indicated in Remark~\ref{rmk:weighted-spaces}, to prove the stability, it suffices to work with a normal vector field $U$ such that both $U$ and $\n^\perp U$ are of polynomial growth, and hence the Fourier coefficients of $U$ satisfy all the weighted assumptions in the analysis obtained in the previous sections, which we now summarize.

    Recall that from Proposition~\ref{prop:form-fourier}, for a normal vector field $U = \frac{u}{g_{\theta\theta}}\N_1 + \frac{v}{g_{rr}}\N_2$ with the decomposition in \eqref{eq:uv-expansion},
    the quadratic form is
    \begin{align*}
    2\sqrt{2} Q(U)= Q_0^\theta(\alpha_0) + Q_0^r(\beta_0)+\sum_{m=1}^\infty \left( Q_m(\alpha^1_m,\beta^2_m) + Q_m(\alpha^2_m,-\beta^1_m) \right).
    \end{align*}
    By Corollaries~\ref{cor:mode-4-analysis},~\ref{cor:mode-3-analysis}, and Proposition~\ref{prop:form-mono}, for $j\ge 0,$ $Q_{3+2j}(\alpha,\beta)\geq Q_3(\alpha,\beta)\ge 0$ for any odd-even pair $(\alpha,\beta)$, and $Q_{4+2j}(\alpha,\beta)\ge Q_4(\alpha,\beta)\geq 0$ for any even-odd pair $(\alpha,\beta).$
    By Propositions~\ref{prop:mode-2-analysis},~\ref{prop:mode-1-analysis},~\ref{prop:mode-0-theta-analysis}, and~\ref{prop:mode-0-r-analysis}, if $U$ is $L^2$-orthogonal to the mean curvature and all the translation directions, then the rest of the low modes are non-negative.
    Hence, $Q(U)\ge 0.$ 
\end{proof}


    
    
    
    

    


As a consequence of the rigidity case in the analysis, we obtain a complete list of unstable modes and Jacobi fields on $M.$

\begin{cor}
    The unstable eigenfunctions of $Q$ are exactly the translations and the dilation in Table~\ref{tab:evectors}.
    The Jacobi fields of $Q$ are exactly the rotations in Table~\ref{tab:evectors}.
\end{cor}

\begin{proof}
    Theorem~\ref{thm:main} implies that any unstable eigenvector must be either a translation direction or the dilation as described in Table~\ref{tab:evectors}.
    By the rigidity cases in Corollaries~\ref{cor:mode-4-analysis},~\ref{cor:mode-3-analysis}, Propositions~\ref{prop:mode-2-analysis},~\ref{prop:mode-1-analysis},~\ref{prop:mode-0-theta-analysis}, and~\ref{prop:mode-0-r-analysis}, any Jacobi field must be a rotation vector field as in Table~\ref{tab:evectors}.
\end{proof}

As a further consequence of Theorem~\ref{thm:main} and \cite{Y16}*{Theorem~9.5}, the first and second non-zero eigenvalues of the drift Laplacian $\mathcal L=\D-\n_{x^T}$ acting on functions on $M$ are $1/2$ and $1.$
The eigenspace corresponding to the eigenvalue $1/2$ is spanned by the coordinate functions.

\subsection{Additional remarks}
\label{sec:remark}

We leave a few remarks in this section.

\begin{remark}
\label{rmk:instability-cyl}
    Theorem~\ref{thm:main} is about the stability of the embedding \eqref{eq:M-embedding} of the M\"obius shrinker.
    If one looks at the immersion \eqref{eq:F-cyl-parametrization} as in the viewpoint in \cite{LW09}, it is likely that one could prove the immersed cylindrical shrinker is $F$-unstable.
    Numerical tests suggest that as a global section of the normal bundle of the immersed cylindrical shrinker, $V=\frac{\sinh(2r)}{g_{\theta\theta}}\N_1$ provides an unstable direction which is orthogonal to the translations and the dilation.
\end{remark}


\begin{remark}
\label{rmk:entropy}
    The entropy of the M\"obius shrinker $M$ can be calculated explicitly by
    \begin{align*}
    \lambda(M)= \frac{\sqrt2\,e^{1/8}}8
    \left( K_0\!\left(\frac38\right) + 3K_1\!\left(\frac38\right)\right)
    \approx 1.65,
    \end{align*}
    where $K_n(z)=\int_0^\infty e^{-z\cosh x}\cosh(n x)\,dx$ denotes the modified Bessel function of the second kind.
    It is worth comparing it with $\lambda(S^2)\approx 1.47,$ $\lambda(S^1\times \bb R)\approx 1.52,$ and the bounds and conjectures proposed in \cite{CM20,B25}.
    As proposed in \cite{B26}, the M\"obius shrinker may have the least entropy among non-orientable shrinkers in $\bb R^4.$ 
\end{remark}

\begin{remark}
\label{rmk:Lawlor}
    In Joyce's program \cite{J15}, at a neckpinch singularity modeled on the union of two planes with the same Lagrangian angle,  one expects a surgery using Lawlor necks or the Joyce--Lee--Tsui expanders as local models.
    We remark that the entropy of such a tangent flow is $2,$ and hence by Remark~\ref{rmk:entropy}, a possible surgery process of undoing the real blow-up described in Section~\ref{sec:intro} may be topologically simpler.
\end{remark}

\begin{remark}
\label{rmk:other-LW-shrinkers}
    The shrinkers constructed in \cite{LW09} are not embedded except the M\"obius shrinker after a quotient (cf. Remark~\ref{rmk:instability-cyl}).
    Determining the (in)stability of the rest of those immersed shrinkers will contribute to the program of classifying possible generic singularities in $\bb R^4.$
\end{remark}

We remark that we had started studying the M\"obius shrinker before we found the construction in \cite{LW09}.
We mention some of our original motivations in the following remarks.

\begin{remark}
    The first-named author learned a construction of the M\"obius shrinker from Jacob Bernstein, who observed in \cite{B25} that flows of unstable perturbations of the Veronese minimal surface could lead to interesting non-orientable shrinkers. Bernstein realized that some conformal deformations of the Veronese inherit some of the symmetry of the Veronese. In particular, there is some subgroup of $SO(5)$ that fixes the Veronese and these conformal deformations are fixed by appropriate one-dimensional subgroups of this group. If one assumes there is a singularity of mean curvature flow that forms one of these $\mathbb{S}^1$ symmetries, then one is led to an ODE which can be solved to give the equation for the self-shrinking M\"obius band.    
\end{remark}


\begin{remark}
\label{rmk:LNZ-construction}
    One of the motivations of this project was to find a mean curvature flow version of the FIK shrinker constructed in \cite{FIK03}.
    With Keaton Naff and Jingze Zhu, the second-named author studied an ansatz of the form $\pr{f_1(r)e^{2i\theta},f_2(r)e^{i\theta}},$ solved the corresponding shrinker ODEs for $f_1$ and $f_2,$ and obtained a conjugated version of the M\"obius shrinker. 
    Along this line, it is particularly interesting to study possible singular behavior of flows within this form of ansatz.
    See, for example, \cite{M14,S15} for the Ricci flow analog.
\end{remark}

\end{document}